\documentclass[11pt,reqno]{article} 

\usepackage[utf8]{inputenc}
\usepackage{graphicx}
\usepackage[all]{xy}
\usepackage{amssymb}
\usepackage{amsmath} 
\usepackage{amsthm}
\usepackage{verbatim}
\usepackage{latexsym}
\usepackage{mathrsfs}

\usepackage{mathrsfs}

\makeindex

\renewcommand{\Cup}{\bigcup}
\renewcommand{\Cap}{\bigcap}

\renewcommand{\a}{\alpha}
\renewcommand{\b}{\beta}
\newcommand{\g}{\gamma}
\newcommand{\e}{\varepsilon}
\renewcommand{\d}{\delta}
\renewcommand{\k}{\kappa}
\renewcommand{\l}{\lambda}

\renewcommand{\o}{\omega}

\newcommand{\es}{\varnothing}

\newcommand{\MM}{\mathcal{M}}

\newcommand{\Po}{\mathcal{P}}             

\newcommand{\ZFC}{{\operatorname{ZFC}}}
\newcommand{\Mod}{{\operatorname{Mod}}}

\newcommand{\Borel}{{\operatorname{Borel}}}

\newcommand{\cf}{\operatorname{cf}}

\newcommand{\id}{\operatorname{id}}

\newcommand{\height}{{\operatorname{ht}}}
\newcommand{\pr}{\operatorname{pr}}
\newcommand{\rest}{\!\restriction\!}

\newcommand{\restl}{\restriction}  

\renewcommand{\le}{\leqslant}

\newcommand{\sd}{\,\triangle\,}             

\newcommand{\Sii}{{\Sigma_1^1}}

\newcommand{\PlOne}{\,{\textrm{\bf I}}}
\newcommand{\PlTwo}{\textrm{\bf I\hspace{-1pt}I}}

\DeclareMathSymbol{\mlq}{\mathord}{operators}{``}
\DeclareMathSymbol{\mrq}{\mathord}{operators}{`'}

\swapnumbers
\newtheorem*{Thm*}{Theorem}
\newtheorem{Thm}{Theorem}
\newtheorem{Lemma}[Thm]{Lemma}
\newtheorem{Cor}[Thm]{Corollary}

\newtheorem{claim}{Claim}[Thm]

\theoremstyle{definition}
\newtheorem{Def}[Thm]{Definition}
\newtheorem{Question}[Thm]{Question}

\theoremstyle{remark}

\newcommand{\proofvpara}{\text{}}

\newenvironment{proofVOf}[1] {\vspace{5pt}\noindent \textit{Proof.}\ignorespaces\renewcommand{\proofvpara}{\text{#1}}}
{\nopagebreak\hspace*{\fill}\mbox{$\square_{\,\proofvpara}$}\\\vspace{-8pt}}

\author{Sy-David Friedman\footnote{Kurt G\"odel Research Center,
    Vienna. The first and third authors wish to thank the FWF (Austrian Science Fund)
    for its support through Einzelprojekt P24654-N25.}, Tapani
  Hyttinen\footnote{Mathematics Department, University of Helsinki.},
  Vadim Kulikov\footnote{Kurt G\"odel Research Center, Vienna.}}
\title{On Borel Reducibility in Generalised Baire Space}

\begin{document}

\maketitle

\begin{abstract}
  In this paper we study the Borel reducibility of Borel equivalence relations on the generalised
  Baire space $\k^\k$ for an uncountable $\k$ with the property $\k^{<\k}=\k$.
  The theory looks quite different from its classical counterpart
  where $\k=\o$, although some basic theorems 
  do generalise.
\end{abstract}

We study the generalisations of classical descriptive set theory of Polish spaces
to the setting where instead of the Baire space $\o^\o$ we look at the 
generalised Baire space $\k^\k$ of all functions from $\k$ to $\k$ where $\k$ is an uncountable
cardinal which satisfies $\k^{<\k}=\k$. The topology on this space is generated by the
basic open sets
$$[p]=\{\eta\in \k^\k\mid \eta\supset p\}$$
where $p\in \k^{<\k}$. The resulting collection of open sets is closed
under intersections of length $<\k$. 
The class of $\k$-Borel sets in this space is the smallest class containing the basic open sets
and which is closed under taking unions and intersections of length~$\k$.

In this paper we often work with spaces of the form $(2^\a)^\b$ for some ordinals $\a,\b\le\k$.
If $x\in (2^\a)^\b$, then technically $x$ is a function $\b\to 2^\a$ 
and we denote by $x_\g=x(\g)$ the value at $\g<\b$.
Thus $x_\g$ is a function $\a\to 2$ for each $\g$ and we denote the value at $\d<\a$ by $x_\g(\d)$.
The lengthier notation for $x\in (2^\a)^\b$ is $(x_\g)_{\g<\b}$ as a $\b$-sequence of functions $\a\to 2$.

We say that a topological space is \emph{$\k$-Baire}, if the intersection of $\k$ many
dense open sets is never empty. The generalised Baire space is $\k$-Baire~\cite{MekVaa}.
If $X$ is a topological space, we say that $A\subseteq X$ is \emph{$\k$-meager} if its complement
contains an intersection of $\k$ many dense open sets. Thus, $X$ is $\k$-Baire if and only if
$X$ is itself not meager (we always drop the prefix ``$\k$-''). 
The complement of a meager set is called
\emph{co-meager}. A set $A\subseteq X$ has the \emph{Baire property} if 
there exists an open set $U$ such that the symmetric difference $U\sd A$ is meager. 
When we write $\forall^*x\in A(P(x))$ we mean 
that there is a co-meager 
set such that every element of that set which belongs to $A$ 
satisfies $P$.
We write $\exists^*x\in A(P(x))$ to mean that there exists a
non-meager set such that every element of that set which belongs to
$A$ satisfies $P$.

\section{Equivalence Relations Induced by a Group Action}
\label{sec:Action}

Suppose $G$ is a topological group. Let $X$ be a Borel subset of $\k^\k$. 
An action $\rho\colon G\times X\to X$ is
\emph{Borel} if it is Borel as a function i.e. inverse images of open sets are Borel.
This action induces an equivalence relation on $X$ in which two elements $x$ and $y$ are equivalent
if there exists $g\in G$ such that $\rho(g,x)=y$. For example, if the action is Borel, then it is easy
to see that this equivalence relation is~$\Sii$.
This equivalence relation is denoted by $E^X_{G,\rho}$, or just $E^X_G$
if the action is clear from the context.

Here are some examples of equivalence relations induced by a Borel action:
\begin{itemize}
\item $\id$ the identity relation.
\item $\id^+$ the jump of identity. This is an equivalence relation on $(2^\k)^\k$
  where $(x_\a)_{\a<\k}$ and $(y_\a)_{\a<\k}$ are equivalent if the sets
  $\{x_\a\mid \a<\k\}$ and $\{y_\a\mid \a<\k\}$ are equal (see Definition~\ref{def:jump}).
  This is not defined as an equivalence relation induced by a Borel action, but is 
  easily seen to be bireducible with $\id^+_*$ which is an equivalence relation
  on $(2^\k)^\k$ where
  $(x_\a)_{\a<\k}$ and $(y_\a)_{\a<\k}$ are equivalent if 
  there exists a permutation
  $s\in S_\k$ ($S_\k$ is the group of all permutations of $\k$) such that $x_\a=y_{s(\a)}$
  for all~$\a$. The latter is induced by a Borel action of $S_\k$.
\item $E_0$ an equivalence relation on $2^\k$ where $(\eta,\xi)\in E_0$ if there exists $\a<\k$
  such that for all $\b>\a$ we have $\eta(\b)=\xi(\b)$.
\item $E_1$ an equivalence relation on $(2^\k)^\k$ where $(x_\a)_{\a<\k}$ and $(y_\a)_{\a<\k}$ are equivalent
  if there exists $\a<\k$ such that for all $\b>\a$ we have $x_\b=y_\b$.
\end{itemize}

Since all the topologies in this paper are closed under intersections of length $<\k$,
we replace ``finite'' by
``less than $\k$'' when referring to product topologies below.\label{page:producttopology} 

\begin{Thm}\label{thm:ActionToE0}
  Let $G$ be a discrete group of cardinality $\le\k$ and let
  it act in a Borel way on a Borel subset $X\subseteq 2^\k$. Let $E^X_G$
  be the (Borel) equivalence relation induced by this action.
  Then $E^X_G\le_B E_0$.
\end{Thm}
\begin{proofVOf}{Theorem~\ref{thm:ActionToE0}}
  The group $G$ acts on $\Po(G)^\k$ coordinatewise by multiplication
  on the right $g\cdot(X_i)_{i<\k}=(X_i g)_{i<\k}$. This gives rise to
  the equivalence relation $E^{P(G)^\kappa}_G$.

  \begin{claim}\label{cl:Ac1}
    $E^X_G\le_B E^{\Po(G)^\k}_G$.
  \end{claim}
  \begin{proofVOf}{Claim~\ref{cl:Ac1}}
    Let $\pi\colon \k\to 2^{<\k}$ be a bijection.
    Let $x\in X$ and for each $\a<\k$ let 
    $$Z_\a(x)=\{g\in G\mid gx\in [\pi(\a)]\}.$$
    This defines a reduction: an element $x\in X$ is mapped to $(Z_\a(x))_{\a<\k}$.
    Suppose there is $g_0\in G$ such that $y=g_0x$ for some $x,y\in X$.
    Then
    \begin{eqnarray*}
      Z_\a(x)&=&\{g\in G\mid gx\in [\pi(\a)]\}\\
             &=&\{gg_0\in G\mid gy\in [\pi(\a)]\}\\
             &=&Z_\a(y)g_0.
    \end{eqnarray*}
    On the other hand suppose that there exists
    $g\in G$ such that $Z_\a(x)=Z_\a(y)g$ for all $\a<\k$. It is enough to show that
    $g^{-1}y\in [p]$ for all basic open neighbourhoods $[p]$ of $x$. So suppose $U=[p]$ is a basic neighbourhood
    containing $x$ and let $\a=\pi^{-1}(p)$. Now obviously $1_G\in Z_\a(x)$, so
    $1_G\in Z_\a(y)g$ and thus $g^{-1}\in Z_\a(y)$, i.e. $g^{-1}y\in [p]$.
  \end{proofVOf}

  For a set $S$, $F_S$ is the free group generated by elements of $S$. 
  $F_{\es}=F_0$ is the trivial group.

  \begin{claim}\label{cl:Ac2}
    $E^{\Po(G)^\k}_G\le_B E^{\Po(F_\k)^\k}_{F_\k}$.
  \end{claim}
  \begin{proofVOf}{Claim \ref{cl:Ac2}}
    Since $G$ has size $\le\k$ and $F_\k$ is a free group on $\k$ generators,
    there is a normal subgroup $N\subseteq F_\k$ such that $G\cong F_\k/N$.
    Assume without loss of generality that $G=F_\k/N$.
    Let $\pr$ be the canonical projection map $F_\k\to F_\k/N$.
    For $(A_\a)_{\a<\k}\in \Po(G)^{\k}$, let
    $$F((A_\a)_{\a<\k})=(\pr^{-1}A_\a)_{\a<\k}.$$
    This is clearly a continuous reduction.    
  \end{proofVOf}

  \begin{claim}\label{cl:Ac3}
    $E^{\Po(F_\k)^\k}_{F_\k}\le_B E_0$.
  \end{claim}
  \begin{proofVOf}{Claim \ref{cl:Ac3}}
    The space $\Po(F_\k)^\k$ can be canonically thought to be the same as
    $(2^\k)^{F_\k}$ the bijection being defined by
    $(A_\a)_{\a<\k}\mapsto x$ where $x(g)(\a)=1$ if and only if
    $g\in A_\a$. This space is equipped with the product topology (recall that
    our definition of product topology is non-standard, see page~\pageref{page:producttopology}): 
    a basic open set is
    given by $$\{x\mid x(g)\rest\a=p\}$$
    for some ordinal $\a<\k$, some $g\in F_\k$ and some $p\in 2^{<\k}$
    and the resulting collection of open sets is
    closed under
    intersections of length~$<\k$. The action of $F_\k$ on $(2^\k)^{F_\k}$ is then defined by
    $g*x=y$ where $y(f)(\a)=1$ if and only if $x(fg^{-1})(\a)=1$ for all $f\in F_\k$.
    
    This space is more convenient for us to work with. Additionally we 
    identify $2^\a$ with $\Po(\a)$ for all $\a$ and write
    $2^\a\subseteq 2^\b$, meaning that an element $p$ of $2^\a$ is identified with $q$ in $2^\b$
    where $q$ is just $p$ with ``$\b-\a$'' zeros at the end.
    
    Let us look at the sets $X_\a=(2^\a)^{F_\a}$ for $\a\le \k$. 
    
    If $w\in X_\b$ and $\a<\b$, denote
    by $w\rest\a$ the element $v$ of $X_\a$ such that $v(g)=w(g)\rest \a$ for all $g\in F_\a$.
    Thus, by the identifications we made, $X_\a\subseteq X_\b\subseteq X_\k$ for all $\a<\b<\k$.
    For every $\a<\k$ fix a well-ordering $<_\a$ of $X_\a$.
    Given $g\in F_\a$ and $w\in X_\a$, let $g*w\in X_\a$ be the element such that
    $(g*w)(f)=w(fg^{-1})$ for all $f\in F_\a$. This is an action of $F_\a$ on $X_\a$ and for $\a=\k$
    it coincides with the original action of $F_\k$ on $\Po(F_\k)^\k$ under the mentioned identifications.
    
    Fix $x\in X_\k$.
    For each $\a$, let $x(\a)$ be the $<_\a$-least element of
    $$\{g*(x\rest\a)\mid g\in F_\a\}$$
    and let $H(x)=(x(\a))_{\a<\k}$. We claim that for all $x,y\in
    X_\k$, 
    $y=g*x$ for some $g\in F_\k$ 
    if and only if there exists $\b<\k$ such that for all $\a>\b$, $x(\a)=y(\a)$.
    
    Assume first that such $g\in F_\k$ exists. Then $g\in F_\b$ for some $\b<\k$ and for all 
    $\a>\b$, $g\in F_\a$. Thus, it is obvious that $x(\a)=y(\a)$ for $\a>\b$, because for these $\a$
    $$\{g*(x\rest\a)\mid g\in F_\a\}=\{g*(y\rest\a)\mid g\in F_\a\}.$$

    Assume now that there exists $\b<\k$ such that for all $\a>\b$, $x(\a)=y(\a)$. Then for each $\a>\b$
    there exists $g_\a\in F_\a$ such that $x\rest\a=g_\a*(y\rest \a)$. For each $\a>\b$, let $\g(\a)$
    be the least ordinal such that $g_\a\in F_{\g(\a)}$. 
    If $\a$ is a limit ordinal, then $\g(\a)<\a$ and so there is $\g_0$ and a stationary $S_0\subseteq \lim\k$
    such that for all $\a\in S_0$ we have $g_\a\in F_{\g_0}$. Since $|F_{\g_0}|<\k$, there is a stationary 
    $S\subseteq S_0$ and $g_*\in F_{\g_0}$ such that for all $\a\in S$ we have $g_\a=g_*$.
    Since $S$ is unbounded, this obviously implies that $y=g_* *x$.

    Fix bijections $f_\a\colon X_\a\to \k$ and map each $x\in X_\k$ to the sequence 
    $(f_\a(x(\a)))_{\a<\k}$ and denote this mapping by $G$. By the above we have
    $x=g*y$ for some $g\in F_\k$ if and only if $(G(x),G(y))\in E_0$.
    It remains to show that $G$ is continuous. 

    Suppose $x\in X_\k$ and take an open neighbourhood $U$ of $G(x)$. Then there is $\b$ such that
    $$\{\eta\in\k^\k\mid \forall\a<\b(\eta(\a)=f_\a(x(\a)))\}\subseteq U.$$
    Now, the set $\{y\in F_\k\mid y\rest\b=x\rest\b\}$ is mapped inside $U$ and contains $x$, so it remains
    to show that this set is open, but this follows from the definition of the topology on $(2^\k)^{F_\k}$ 
    in particular that the collection of open sets is closed under intersections of length~$<\k$.
  \end{proofVOf}
\end{proofVOf}

\begin{Thm}[$V=L$]
  There is a Borel equivalence relation $E$ whose classes have size
  $2$, which is smooth (i.e. Borel reducible to $\id$) yet which is
  not induced by a Borel action of a group of size~$\le\k$.
\end{Thm}
\begin{proof}
  \begin{claim}
    There is an open dense set $O\subseteq 2^\k$ and a bijection $f\colon O\to 2^\k\setminus O$
    such that the graph of $f$ is Borel, but $f$ is not Borel as a function on any non-meager Borel set. 
    However the inverse of $f$ is Borel.
  \end{claim}
  \begin{proof}
    We let $O$ be the complement of a certain closed set of ``master codes'' for
    size $\kappa$ initial segments of $L$.
    This is defined as follows. Let
    $\cal L$ be the language of set theory augmented by constant symbols
    $\bar\alpha$ for each ordinal $\alpha<\kappa$. Also let $T_0$
    denote the theory $\ZFC^-$ ($\ZFC$ minus the power set axiom) 
    plus $V=L$ plus the statement 
    ``there are only boundedly many ordinals $\beta$ such that $L_\beta$
    satisfies $\ZFC^-$''. We consider complete, consistent theories $T$ which
    extend $T_0$ and which in addition satisfy the following:
    
    \bigskip
    
    \noindent
    1. There is no $\omega$-sequence of formulas $\varphi_n(x)$
    (mentioning constants $\bar\alpha$ for $\alpha<\kappa$) such that for
    each $n$ both the sentence ``$\exists$!$x \ \varphi_n(x)$'' and the 
    sentence ``$\exists x,y (\varphi_n(x)\wedge\varphi_{n+1}(y)\wedge y\in
    x)$'' belong to $T$.\\
    2. For each $\beta<\kappa$ and formula $\varphi(x)$ (mentioning
    constants $\bar\alpha$ for $\alpha<\kappa$) if the sentences
    ``$\exists$!$x\varphi(x)$'' and ``$\exists x (\varphi(x)\wedge x
    <\bar\beta)$'' both belong to $T$ then so does the sentence
    ``$\exists x(\varphi(x)\wedge x=\bar\gamma)$'' for some
    $\gamma<\beta$.
    
    \bigskip
    
    By identifying sentences of $\cal L$ with ordinals less than $\kappa$
    we can regard theories in $\cal L$ as subsets of $\kappa$. Now let
    $C\subseteq 2^\kappa$ be the set of theories $T$ as above. Then $C$ is
    a closed set. And $C$ is nowhere dense as any set of $\cal
    L$-sentences of size less than $\kappa$ is included in an inconsistent
    such set. 
    
    The theories in $C$ are exactly the first-order theories of structures
    of the form $L_\beta$ in which the constant
    symbol $\bar\alpha$ is interpreted as the ordinal $\alpha$ for each
    $\alpha<\kappa$ and in
    which the axioms of $T_0$ hold. And for each $T$ in $C$ there is a
    unique such model $M(T)$ in which every element is definable from
    parameters less than $\kappa$. 
    
    We are ready to define the function
    $f:O\to C$ where $O$ is the complement of $C$ in $2^\kappa$. List the
    elements of $O$ in $<_L$-increasing order as $x_0,x_1,\ldots$ and list
    the elements of $C$ in $<_L$-increasing order as
    $y_0,y_1,\ldots$; then we set $f(x_i)=y_i$ for each $i<\kappa^+$.
    The inverse of $f$ is Borel because given $y\in C$ we can identify 
    $f^{-1}(y)$ (viewed as a subset of $\kappa$) as the set of $\gamma<\kappa$
    such that the sentence ``$\gamma$ belongs to the $i$-th element
    in the $<_L$-increasing enumeration of $2^\kappa$ where $i$ is
    the order type of the set of $\beta$ such that $L_\beta$ models
    $T_0$'' belongs to (the theory associated to) $y$. Thus the graph
    of $f$ is Borel. If $B$ is a non-meager Borel set and $g$ is a Borel
    function then we claim that $g$ cannot agree with $f$ on $B$: Indeed,
    let $\beta_0$ be so that $L_{\beta_0}$ models $\ZFC^-$ and 
    contains Borel codes for both $B$ and $g$ and let $x\in B\cap O$ be
    $\kappa$-Cohen generic over $L_{\beta_0}$. Then $f(x)=T$ is the theory of a
    model $M(T)=L_\beta$ where $\beta$ is greater than $\beta_0$. But
    $g(x)$ belongs to $L_{\beta_0}[x]$, which by the genericity of $x$ is
    a model of $\ZFC^-$ while $L_{\beta_0}[f(x)]$ does not satisfy
    $\ZFC^-$ as $f(x)=T$ codes the model $L_\beta$. 
  \end{proof}
  Define $xEy$ if and only if $x=y$, $y=f(x)$ or $x=f(y)$. 
  Now $E$ has a Borel transversal, i.e., there is a Borel function $t$
  such that $xEt(x)$ for all $x$ and $xEy$ if and only if $t(x)=t(y)$ for all $x,y$:
  Given $x\in 2^\k$, first decide in a Borel
  way if $x$ is in $O$ or not. If yes, then let $t(x)=x$, otherwise,
  find $f^{-1}(x)$ in a Borel way (since $f^{-1}$ is Borel) and let
  $t(x)=f^{-1}(x)$. This $t(x)$ is a Borel transversal. It follows
  that $E$ is smooth.
  
  Finally, suppose $E$ is given by a Borel action of some
  group $G$ of size at most $\kappa$. Then for each $x\in O$ choose
  $g_x\in G$ such that $f(x)=g_x\cdot x$; then for some fixed $g\in
  G$, $f(x)=g\cdot x$ for non-meager many $x\in O$, contradicting the
  fact that $f$ is not Borel on any non-meager Borel set.
\end{proof}

\begin{Question}
  Is there a Borel equivalence relation with classes of size $\k$ which is not
  reducible to~$E_0$?
\end{Question}

\section{$E_1$ and $E_{club}$}
\label{sec:E1Eclub}

Let $E_1$ be the equivalence relation on $(2^\k)^\k$ where $(x_\a)_{\a<\k}$ and $(y_\a)_{\a<\k}$
are equivalent if there exists $\b<\k$ such that for all $\g>\b$, $x_\g=y_\g$.

\begin{Thm}
  $E_1$ and $E_0$ are bireducible.
\end{Thm}
\begin{proof}
  It is obvious that $E_0\le_B E_1$, so let us look at the other direction.

  To simplify notation, we think of $E_{0}$ on $\k^{\k}$: two functions $\eta$
  and $\xi$ are $E_0$-equivalent if the set $\{\a<\k\mid \eta(\a)\ne\xi(\a)\}$
  is bounded. It is easy to see that $E_0$ on $2^\k$ is bireducible with this
  equivalence relation.

  For all limit $\a <\k$, define $E^{\a}_{1}$ to be the equivalence relation
  on $(2^{\a})^{\a}$ approximating $E_1$
  i.e. $(x_{i})_{i<\a}E^{\a}_{1}(y_{i})_{i<\a}$
  if for some $\b <\a$, $x_{i}=y_{i}$ for all $i>\b$.
  Now define the reduction $F\colon (2^{\k})^{\k}\rightarrow\k^{\k}$
  so that for all $(x_{i})_{i<\k}\in (2^{\k})^{\k}$,
  $F((x_{i})_{i<\k})(\a)=0$ if $\a$ is not a limit and otherwise
  it is a code for the $E_1^\a$-equivalence class of $(x_{i}\rest\a)_{i<\a}$.

  Clearly $F$ is continuous and if
  $(x_{i})_{i<\k}E_{1}(y_{i})_{i<\k}$, then also
  $F((x_{i})_{i<\k})$ and $F((y_{i})_{i<\k})$
  are $E_{0}$-equivalent (if $\b<\k$ witnesses the first
  equivalence, it witnesses also the second).

  Also if $(x_{i})_{i<\k}$ and $(y_{i})_{i<\k}$ are not
  $E_{1}$ equivalent, then for all $\a <\k$ there are $\g ,\b <\k$
  such that $\b >\a$ and $x_{\b}(\g )\ne y_{\b}(\g )$.
  Let $f(\a )$ be $max\{ \b ,\g\}$. Now if $\a^{*}<\k$
  is such that for all $\a <\a^{*}$, $f(\a )<\a^{*}$,
  then clearly $(x_{i}\rest\a^{*})_{i<\a^{*}}$ and
  $(y_{i}\rest\a^{*})_{i<\a^{*}}$ are not $E^{\a^{*}}_{1}$-equivalent,
  and thus $F((x_{i})_{i<\k})(\a^{*})\ne F((y_{i})_{i<\k})(\a^{*})$.
  Since the set of such $\a^{*}$ is unbounded,
  $F((x_{i})_{i<\k})$ and $F((y_{i})_{i<\k})$ are not
  $E_{0}$-equivalent. 
\end{proof}

\begin{Def}\label{def:jump}
  If $E$ is an equivalence relation on $2^\k$, its \emph{jump} is the equivalence
  relation denoted by $E^+$ on $(2^\k)^\k$ defined as follows. Two sequences
  $(x_\a)_{\a<\k}$ and $(y_\a)_{\a<\k}$ are $E^+$-equivalent, if 
  $$\{[x_\a]_{E}\mid \a<\k\}=\{[y_\a]_E\mid \a<\k\}$$
  where $[x]_E$ is the equivalence class of $x$ in $E$. Since
  $(2^\k)^\k$ is homeomorphic to $2^\k$ we can assume without loss of generality
  that $E^+$ is also defined on $2^\k$.

  For an ordinal $\a<\k^+$ define $E^{\a+}$ by transfinite induction.
  To begin, define $E^{0+}=E$.
  If $E^{\a+}$ is defined, then $E^{(\a+1)+}=(E^{\a+})^+$. 

  Suppose $\a$ is a limit and $E^{\b+}$
  is defined to be an equivalence relation on $2^\k$ for $\b<\a$.
  Let $X$ be the disjoint union of $\a$ many copies of $2^\k$. 
  Denote the $\b$:th copy
  by $X_\b$, thus $X=\Cup_{\b<\a}X_\b$. Let $h$ be a homeomorphism $X\to 2^\k$.
  Two functions $\eta$ and $\xi$ are defined to be $E^{\a+}$-equivalent, 
  if $h^{-1}(\eta)$ and $h^{-1}(\xi)$
  belong both to the same $X_\b$ and are $E^{\b+}$-equivalent.
  This is called the \emph{join} of the equivalence relations
  $\{E^{\b+}\mid \b<\a\}$ and is denoted $\bigoplus_{\b<\a}E^{\b+}$.
\end{Def}

\begin{Thm}\label{thm:idplnottoe1}
  $E_0<_B \id^+$
\end{Thm}
\begin{proof}
  The reduction
  is defined by
  $$E_0\le_B \id^+\colon\ \eta\mapsto (p+\eta)_{p\in 2^{<\k}}.$$

  Suppose $f\colon 2^\k\to (2^\k)^\k$ is a Borel reduction from $\id^+$ to~$E_0$.
  There is a co-meager set $D$ on which $f$ is continuous. Without
  loss of generality assume that this $D$ is the intersection
  $\Cap_{i<\k}D_i$ where $D_i$ are dense open.

  For every $i<\k$ we will define ordinals $\g_i$ together with
  sequences $x^i=(x^i_\a)_{\a<\g_i}$ and $y^i=(y^i_\a)_{\a<\g_i}$ where each $x^i_\a,y^i_\a\in 2^{\g_i}$ 
  and permutations $\pi_i\in S_{\g_i}$. These will satisfy the following
  requirements for every $i<j<\k$:
  \begin{enumerate}
  \item $\pi_i\subseteq \pi_j$.
  \item $\g_i\le \g_j$,
  \item For all $\a<\g_i$ we have $x^i_\a\subseteq x^j_\a$ and $y^i_\a\subseteq y^j_\a$.
  \item For all $\a<\g_i$ we have $x^i_\a=y^i_{\pi_i(\a)}$.
  \item Let $[(x^i_\a)_{\a<\g_i}]$ be the set of all $x=(x_\a)_{\a<\k}\in (2^\k)^\k$ 
    such that $x^i_\a\subseteq x_\a$ for all $\a$. There exist $\b>i$, $\d<\k$ and $p,q\in (2^{\d})^{\b+1}$
    such that 
    $$f[[(x^i_\a)_{\a<\g_i}]\cap D]\subseteq [p],$$
    $$f[[(y^i_\a)_{\a<\g_i}]\cap D]\subseteq [q]$$
    and $p(\b)\ne q(\b)$.
  \item $[(x^{i+1}_{\a})_{\a<\g_{i+1}}]\subseteq D_i$ and $[(y^{i+1}_{\a})_{\a<\g_{i+1}}]\subseteq D_i$
  \end{enumerate}
  This will lead to a contradiction as follows. Let $\tilde x = (\tilde x_\a)_{\a<\k}$ be such that
  for every $\a$ we have $\tilde x_\a\rest\g_i=x^i_\a$ if $\g_i>\a$. This is possible
  by (2) and (3). Analogously define $\tilde y$. Now by (1) we can define
  $\pi=\Cup_{i<\k}\pi_i$ which by (4) witnesses that $\tilde x$ and $\tilde y$
  are $\id^+$-equivalent. By (6) they are in $D$ and by continuity in 
  $D$ and by (5) the images $f(\tilde x)$ and $f(\tilde y)$
  cannot be $E_0$-equivalent.

  Let $x^*=(x^*_\a)_{\a<\k}$ and $y^*=(y^*_\a)_{\a<\k}$ be any sequences in $D$ such that $x^*$ is not $\id^+$-equivalent to $y^*$.
  Find these for example as follows: We will define sequences $(\xi_k)_{k<\k}$ and $(\eta_k)_{k<\k}$
  and ordinals $\e_k$ such that for all $k<\k$ we have $\xi_k,\eta_k\in (2^{\e_k})^{\e_k}$, for $k_1<k_2$
  we have $\e_{k_1}<\e_{k_2}$, $\xi_{k_1}\subseteq\xi_{k_2}$ and $\eta_{k_1}\subseteq\eta_{k_2}$, and
  the unions $\Cup_{k<\k}\xi_k$ and $\Cup_{k<\k}\eta_k$
  are in $D$ and not $\id^+$-equivalent. This is easy: Let $\e_0=0$, $\xi_0=\es$ and $\eta_0=\es$. If $\xi_{k}$ and $\eta_k$
  are defined, first extend $\xi_k$ to an element $\xi_{k+1}'\in (2^{\e_{k+1}'})^{\e_{k+1}'}$  
  (for suitable $\e_{k+1}'>\e_{k}$) such that $[\xi_{k+1}']\subseteq D_k$.
  Then extend the first component of $\eta_k$ so that it differs in a diagonal way from every component of $\xi_{k+1}'$.
  After that, extend the result into $\eta_{k+1}\in (2^{\e_{k+1}})^{\e_{k+1}}$ (for suitable $\e_{k+1}>\e_{k+1}'$) so that 
  $[\eta_{k+1}]\subseteq D_k$ and $\e_{k+1}>\e_{k+1}'$.
  Finally extend $\xi_{k+1}'$ to an element of $(2^{\e_{k+1}})^{\e_{k+1}}$ so that the first component of $\eta_{k+1}$ is
  still diagonally different from every component of $\xi_{k+1}$; technically this means that
  $\eta_{k+1}(0)(\a)\ne \xi_{k+1}(\a)(\a)$. 
  At limit $k$ just take the natural limits of the sequences. In this way at the
  $\k$:th limit we obtain $\xi_\k$ and $\eta_\k$ are as required, so
  we can define $x^*=\xi_k$ and $y^*=\eta_\k$. 

  Let $\b$ and $\d$ be such that $f(x^*)(\b)(\d)\ne f(y^*)(\b)(\d)$ 
  which exist because $f$ is assumed to be a reduction
  and $f(x^*)$ and $f(y^*)$ are not $E_0$-equivalent. Now by continuity in $D$ 
  there is $\g_0^*>0$ such that
  $$f[[(x^*_\a\restl\g^*_0)_{\a<\g^*_0}]\cap D]\subseteq [(f(x^*)\rest(\b+1)]$$
and
  $$f[[(y^*_\a\restl\g^*_0)_{\a<\g^*_0}]\cap D]\subseteq [(f(y^*)\restl(\b+1)]$$

  Then we glue $(x^*_\a\rest\g_0^*)_{\a<\g^*_0}$ to the end of $(y^*_\a\rest\g_0^*)_{\a<\g^*_0}$ and vice versa:

  Let $\g_0=\g^*_0+\g^*_0$ and
  for all $\a<\g^*_0$ define 
  \begin{eqnarray*}
      x^{0}_{\a}&=&x^*_{\a}\rest \g_0\\
      y^{0}_{\a}&=&y^*_{\a}\rest \g_0\\
      x^{0}_{\g^*_0+\a}&=&y^*_{\a}\rest\g_0\\
      y^{0}_{\g^*_0+\a}&=&x^*_{\a}\rest\g_0
  \end{eqnarray*}
  Let $\pi_0$ be the permutation which takes $\a$ to $\g_0+\a$ when $\a<\g_0$ and if $\a=\g_0+\e$,
  then $\pi(\a)=\e$.
  So we have defined $\pi_0$, $\g_0$, $(x^{0}_\a)_{\a<\g_0}$ and $(y^{0}_\a)_{\a<\g_0}$ such that
  all the conditions (1)--(6) are satisfied so far.

  Suppose that $\pi_i$, $\g_i$, $(x^i_\a)_{\a<\g_i}$
  and $(y^i_\a)_{\a<\g_i}$ are defined for $i<j$ such that the conditions (1)--(6) are satisfied.
  If $j$ is a limit, then just define $\pi_j=\Cup_{i<j}\pi_i$, $x^j_\a=\Cup_{i'<i<j}x^i_\a$ 
  and $y^j_\a=\Cup_{i'<i<j}y^i_\a$ for some $i'$ such that $\g_{i'}>\a$ and $\g_j=\sup_{i<j}\g_i$
  and $\g_j=\sup_{i<j}\g_i$. 

  Suppose $j$ is a successor, in fact w.l.o.g denote the predecessor by $i$, 
  i.e. $j=i+1$. Next we want to build elements $x^*=(x^*_\a)_{\a<\k}$ in $[(x_i)_{i<\g_i}]\cap D$
  and  $y^*=(y^*_\a)_{\a<\k}$ in $[(y_i)_{i<\g_i}]\cap D$ such that $x^*_\a=y^*_{\pi(\a)}$ for all $\a<\g_i$
  and which are not $\id^+$-equivalent. To do that, define 
  $\e_0=\g_i$, $\xi_0=(x^i_\a)_{\a<\g_i}$ and $\eta_0=(y^i_\a)_{\a<\g_i}$.
  Suppose we have defined $\xi_k$ and $\eta_k$ for some $k<\k$. 

  First we extend $\xi_k$ to $\xi_{k+1}'\in (2^{\e_{k+1}'})^{\e_{k+1}'}$ for
  some suitable $\e_{k+1}'<\k$, $\e_{k+1}'>\e_k$
  so that $[\xi_{k+1}']\subseteq D_k$. Then we extend $\eta_k$ first to
  a $\eta_{k+1}'\in (2^{\e_{k+1}'})^{\e_{k+1}'}$ such that $\eta_{k+1}'\rest\g_i$
  equals to the action of $\pi_i$ applied to $\xi_{k+1}'\rest\g_i$ and $\eta_{k+1}'\rest\{\g_i\}$
  diagonally differs from every component of $\xi_{k+1}'$. Then extend $\eta_{k+1}'$
  to $\eta_{k+1}\in (2^{\e_{k+1}})^{\e_{k+1}}$ (for a suitable $\e_{k+1}>\e_{k+1}'$)
  so that $[\eta_{k+1}]\subseteq D_k$. Finally extend $\xi_{k+1}'$
  to $\xi_{k+1}\in (2^{\e_{k+1}})^{\e_{k+1}}$ in any such way that
  $\eta_{k+1}\rest\{\g_i\}$ differs from every component of $\xi_{k+1}$ in a diagonal way.

  At limit $k$ just take the natural limits of the sequences. In this way at the
  $\k$:th limit we obtain $\xi_\k$ and $\eta_\k$ which are as required, so
  we can define $x^*=\xi_k$ and $y^*=\eta_\k$. 

  Now by continuity and by the fact that $x^*$ and $y^*$ are not $\id^+$-equivalent,
  find $\g_{i+1}^*$ and $\b>i+1$ so that 
  $f(x^*)(\b)\ne f(y^*)(\b)$ and
  $$f[[(x^*_\a\restl\g^*_{i+1})_{\a<\g^*_{i+1}}]\cap D]\subseteq [(f(x^*)\restl(\b+1)]$$
  and
  $$f[[(y^*_\a\restl\g^*_{i+1})_{\a<\g^*_{i+1}}]\cap D]\subseteq [(f(y^*)\restl(\b+1)].$$
  Also we make sure that $\g^*_{i+1}$ is big enough
  so that (6) is satisfied. Now we want to glue a part of
  $(x^*_\a\rest\g_{i+1}^*)_{\a<\g^*_{i+1}}$ to the end of $(y^*_\a\rest\g_{i+1}^*)_{\a<\g^*_{i+1}}$ and vice versa:
  Let $\e=\g_{i+1}-\g_i$, i.e. the order type of $\g_{i+1}^*\setminus \g_i$ and let $\g_{i+1}=\g_{i+1}^*+\e$. 
  Define $x^{i+1}_{\a}$ and $y^{i+1}_\a$ for all $\a<\g_{i+1}$ depending on $\a$ as follows.
  If $\a<\g_{i+1}^*$, let $x^{i+1}_\a$ to be $x^*_\a\rest\g_{i+1}$ and $y^{i+1}_\a$ to be $y^*_\a\rest\g_{i+1}$.
  If $\a=\g_{i+1}^*+\d$ for some $\d<\e$, then let $x^{i+1}_\a$ to be $y^*_{\g_i+\d}$ and $y^{i+1}_{\a}$ to be 
  $x^*_{\g_i+\d}$. This gives us also $\pi_{i+1}$ and we are done.
\end{proof}

\begin{Def}\label{def:E_Cub}
  For a regular cardinal $\mu<\k$ and $\l\in\{2,\k\}$ let $E^\l_{\mu\text{-cub}}$ be the equivalence
  relation on $\l^\k$ such that $\eta$ and $\xi$ are
  $E^\l_{\mu\text{-cub}}$-equivalent if the set $\{\a\mid \eta(\a)=\xi(\a)\}$
  contains a $\mu$-cub, i.e. an unbounded set which is closed under
  $\mu$-cofinal limits. If $T$ is a countable complete first-order
  theory, denote by $\cong^T_\k$ the isomorphism relation on the
  models of~$T$.
\end{Def}

In the following we show that 
\begin{enumerate}
\item The $\a$:th jump of identity for $\a<\k^+$ is reducible to $E^\k_{\mu\text{-cub}}$
  for every regular $\mu<\k$,
\item Every Borel isomorphism relation is reducible to $E^\k_{\mu\text{-cub}}$ for
  every regular $\mu<\k$,
\item If $T$ a countable complete first-order classifiable 
  (superstable with NDOP and NOTOP) and 
  shallow theory, then $\cong^\k_T\,\le_B\, E^\k_{\mu\text{-cub}}$.
\end{enumerate}

\begin{Def}
  Fix a limit ordinal $\a\le \k$ and let $t$ be a subtree of $\a^{<\o}$ with 
  no infinite branches.
  Let $h$ be a function from the leaves of $t$ to $2^{<\a}$. Then $(t,h)$
  determines the set $B_{(t,h)}$ as follows: $p\in 2^\a$ belongs to
  $B_{(t,h)}$ if player $\PlTwo$ has a winning strategy in the game
  $G(p,t,h)$: The players start at the root and then one after another 
  choose a successor of the node they are in and then move to that
  successor. Player $\PlOne$ starts.
  Eventually they reach a leaf $l$ and player $\PlTwo$ wins
  if $h(l)\subset p$. We say that $(t,h)$ is a \emph{Borel code for $\a$}.

  If $\a=\k$,
  it is easy to see by induction on the rank of the tree that $B_{(t,h)}$
  is a usual Borel set and conversely, if $B\subset 2^{\k}$ 
  is any Borel set, then there is
  a Borel code $(t,h)$ for $\k$ such that $B=B_{(t,h)}$. 

  If $t$ is replaced by a more general $\k^+\k$-tree (subtree of $\k^{<\k}$ 
  without branches of length~$\k$), then the sets that
  are obtained in this way
  are the so called $\Borel^*$ sets, see \cite{Bl,MekVaa,Hal,FHK}.

  Suppose $(t,h)$ is a Borel code for $\k$ and $\a<\k$. 
  Say that $\a$ is \emph{good} for $(t,h)$, if for all leaves $l\in t$
  with $\height(l)<\a$ we have $h(l)\in 2^{<\a}$. Clearly the set of good
  $\a$ for a fixed $(t,h)$ is a cub set.
  
  Define the $\a$:th \emph{approximation} of $(t,h)$, denoted $(t,h)\rest\a$
  to be the pair $(t\rest\a,h\rest\a)$ where $t\rest\a=t\cap \a^{<\o}$
  and for all leaves $l$ of $t\rest\a$, $(h\rest\a)=h\rest (t\rest\a)$.
  It is obvious that if $(t,h)$ is a Borel code for $\k$ and $\a<\k$
  is good for $(t,h)$,
  then $(t,h)\rest\a$ is a Borel code for~$\a$.

  By replacing $2^{<\a}$ by $(2^{<\a})^2$ for the range of $h$
  and making necessary changes we
  can define Borel codes for subsets of $(2^\a)^2$.

  Note that the game $G(p,t,h)$ is determined for all $p\in 2^{\a}$
  (this is not the case for general $\Borel^*$-sets).
\end{Def}

Make a similar definition for codes of Borel subsets of 
$2^\k \times 2^\k$.

\begin{Lemma}\label{lem:BorelApprox}
  Suppose that $B = B_{(t,h)}$ is a Borel subset of $2^\k \times 2^\k$. 
  Then
  $$(\eta,\xi)\in B\iff (\eta\rest\a,\xi\rest\a)\in B_{(t,h)\restl\a}$$ 
  for cub-many $\a$ and 
  $(\eta,\xi)\notin B \iff (\eta\rest\a,\xi\rest\a)\notin B_{(t,h)\restl\a}$ 
  for cub-many $\a$.
\end{Lemma}
\begin{proof}
  Suppose $(\eta,\xi)\in B$ and let $\sigma$ be a winning strategy of
  player $\PlTwo$ in $G((\eta,\xi),t,h)$. Let $C$ be the set of those
  limit $\a$ which are good for $(t,h)$ and 
  that $t\rest\a$ is closed under $\sigma$. Clearly
  $(\eta\rest\a,\xi\rest\a)\in B_{(t,h)\restl\a}$ for all $\a\in C$
  and $C$ is cub.

  Conversely, if $(\eta,\xi)\notin B$, then player $\PlOne$ has a winning 
  strategy $\tau$ in $G((\eta,\xi),t,h)$ and by closing under $\tau$ we obtain
  the needed cub set again.
\end{proof}

\begin{Lemma}\label{lem:ClassOfEqRels}
  Let $S$ be the set of Borel equivalence relations $E$ such that for some
  Borel code $(t,h)$, $E = B_{(t,h)}$ and $B_{(t,h)\restl\a}$ 
  is an equivalence
  relation for cub-many $\a < \k$. Then $S$ contains $\id$ and is closed
  under jump and the join operation $\bigoplus$ as in the definition of
  iterated jump, Definition~\ref{def:jump}.
\end{Lemma}
\begin{proof}
  Enumerate $2^{<\k}=\{p_\a\mid \a<\k\}$. Let $t=\k^1$ and 
  $h(\a)=(p_\a,p_\a)$. Clearly $\id=B_{(t,h)}$ and for those $\a$
  for which $\{p_i\mid i<\a\}=2^{<\a}$, $B_{(t,h)\restl\a}$ is the
  identity on $2^\a$ and this is clearly a cub set.
  
  Suppose $E$ is in $S$ and $(t,h)$ is a code for $E$ witnessing that
  and that $C$ is a cub set on which $E_{(t,h)}$ is an equivalence relation.
  It is not difficult to design a Borel code $(t^+,h^+)$ for the jump $E^+$
  and check that for cub many $\a\in C$, 
  $B_{(t^+,h^+)\restl\a}$ is the jump of $B_{(t,h)\rest\a}$.
 
  Similarly suppose that $E_i\in S$ are equivalence relations for $i<\k$
  and witnessing codes $(t_i,h_i)$ are given with $C_i$ cub sets 
  such that $B_{(t_i,h_i)\rest\a}$ is an equivalence relation for each
  $\a\in C_i$. Then
  it is not difficult to design a code $(t,h)$ so that $B_{(t,h)}$
  is $\bigoplus_{i<\k}E_i$ and for cub many $\a\in \nabla_{i<\k}C_i$,
  $B_{(t,h)\restl\a}$ is $\bigoplus_{i<\a}B_{(t_i,h_i)\restl\a}$
\end{proof}

It follows that $S$ contains all iterates of the jump $id^{+\b}$,
$\b < \k^+$.

\begin{Thm}
  Let $E$ be an equivalence relation in $S$. Then $E$ is reducible to
  $E^\k_{\mu\text{-cub}}$ for any regular $\mu < \k$ 
  (see Definition~\ref{def:E_Cub}).
\end{Thm}
\begin{proof}
  Let $E$ be $B_{(t,h)}$ where $(t,h)$ witnesses that 
  $E$ belongs to $S$. To each $\eta$
  assign the function $f_\eta$ where $f_\eta(\a)$ is a code for the
  $B_{(t,h)\restl\a}$ equivalence class of $\eta\rest\a$ 
  (if $\cf(\a) = \mu$ and
  $B_{(t,h)\restl\a}$ is an equivalence relation, $0$ otherwise). 
  By Lemma \ref{lem:BorelApprox}, if $\eta E \xi$ then
  $f_\eta(\a) = f_\xi(\a)$ for $\mu$-cub-many $\a$ and if $\lnot \eta E \xi$ then
  $f_\eta(\a) \neq f_\xi(\a)$ for $\mu$-cub-many~$\a$.
\end{proof}

\begin{Cor}
  The iterated jumps $id^{\a+}$ of the identity are reducible to
  $E^\k_{\mu\text{-cub}}$ for each regular $\mu < \k$. \qed
\end{Cor}

\begin{Cor}\label{cor:Borel}
  If $\MM$ is a Borel class of models such that $\cong_\MM$, the isomorphism
  relation on $\MM$ is Borel, then $\cong_{\MM}$ is Borel reducible to 
  $E^\k_{\mu\text{-cub}}$ for all regular $\mu<\k$. 
\end{Cor}
\begin{proof}
  Using similar techniques as in classical descriptive set theory
  (see e.g. \cite[Lemma 12.2.7]{Gao}) one can show that a Borel isomorphism
  can be reduced to an iterated jump of identity. 
\end{proof}

\begin{Cor}
  Suppose $T$ a countable complete first-order classifiable 
  (superstable with NDOP and NOTOP) and 
  shallow theory, then $\cong^\k_T\,\le_B E^\k_{\mu\text{-cub}}$.
\end{Cor}
\begin{proof}
  By \cite[Theorem 68]{FHK} the isomorphism relation of a classifiable
  shallow theory is Borel, so we apply Corollary \ref{cor:Borel}.
\end{proof}

We have shown in
\cite[Theorem 75]{FHK} that
under certain cardinality assumptions on $\k$, a complete countable
first-order theory $T$ is classifiable
if and only if for all regular $\mu<\k$, $E^2_{\mu\text{-cub}}\not\le_B \ \cong_T$,
see Definition~\ref{def:E_Cub},
where $\cong_T$ is the isomorphism on $\Mod(T)$.
Clearly $E^2_{\mu\text{-cub}}\le_B E^\k_{\mu\text{-cub}}$.

\begin{Question}\label{q:CUBreducibility}
  Is $E^\k_{\mu\text{-cub}}$ reducible to $E^2_{\mu\text{-cub}}$?
\end{Question}

If the answer to Question~\ref{q:CUBreducibility} is ``yes'' then using
\cite[Theorem 75]{FHK} we obtain: 
Suppose $T_1$ and $T_2$ are complete first-order theories with
$T_1$ classifiable and shallow and $T_2$ non-classifiable.
Also suppose that $\k=\l^+=2^\l>2^\o$ where $\l^{<\l}=\l$. Then
$\cong_{T_1}$ is Borel reducible to~$\cong_{T_2}$.

\bibliography{ref}{}
\bibliographystyle{alpha}

\end{document}